\documentclass{amsart}
\usepackage{amsmath,amssymb,amsfonts}
\usepackage[mathscr]{eucal}
\usepackage{graphicx,color}
\usepackage{enumerate}
\usepackage[matrix,cmtip,arrow,curve]{xy}
\usepackage[unicode]{hyperref}
\usepackage{tikz-cd}

\theoremstyle{plain}
\newtheorem{theorem}[equation]{Theorem}
\newtheorem{cor}[equation]{Corollary}
\newtheorem{prop}[equation]{Proposition}
\newtheorem{lemma}[equation]{Lemma}

\theoremstyle{definition}
\newtheorem{definition}[equation]{Definition}
\newtheorem{example}[equation]{Example}

\newtheorem{remark}[equation]{Remark}

\numberwithin{equation}{section}

\input xy
\xyoption{all}

\newcommand{\Deltaop}{\Delta^{op}}

\newcommand{\Hom}{\text{Hom}}

\newcommand{\SSets}{\mathcal{SS}ets}

\newcommand{\Sets}{\mathcal Sets}

\newcommand{\ob}{\text{ob}}

\newcommand{\id}{\text{id}}
\newcommand{\Tgam}{\mathcal T_{G\mathcal{AM}}}
\newcommand{\Gammaop}{\Gamma^{op}}
\newcommand{\GGammaop}{G\Gamma^{op}}
\newcommand{\Alg}{\mathcal Alg}
\newcommand{\Algt}{\mathcal Alg^\mathcal T}
\newcommand{\sk}{\text{sk}}

\newcommand{\Fin}{\mathcal Fin}
\newcommand{\Deltaactgamma}{\Delta \circlearrowright \Gamma}

\begin{document}

\title{Diagrams encoding group actions on $\Gamma$-spaces}

\author[J.E.\ Bergner and P.\ Hackney]{Julia E. Bergner and Philip Hackney}

\address{Department of Mathematics, University of California, Riverside, CA 92521}
\email{bergnerj@member.ams.org}

\address{Matematiska institutionen\\ Stockholms Universitet\\ 106 91 Stockholm \\ Sweden}
\email{hackney@math.su.se}



\thanks{The first-named author was partially supported by NSF grant DMS-1105766.}

\begin{abstract}
We introduce, for any group $G$, a category $G\Gamma$ such that diagrams $G\Gamma \rightarrow \SSets$ satisfying a Segal condition correspond to infinite loop spaces with a $G$-action. We also consider diagrams which encode group actions on infinite loop spaces where the group may vary.
\end{abstract}

\maketitle

\date{\today}

\section{Introduction}

In \cite{inverses} and \cite{simpmon}, we looked at diagrammatic ways to approach simple algebraic structures as diagrams satisfying some kind of Segal condition. The terminology comes from the fact that such conditions were first investigated by Segal \cite{segal}. In \cite{simpmon}, we show that diagrams $X \colon \Deltaop \rightarrow \SSets$ such that $X_0 = \Delta[0]$ can be regarded a simplicial monoids when the Segal condition holds either strictly or up to homotopy. In contrast, Segal shows that for diagrams $Y \colon \Gammaop \rightarrow \SSets$ with $Y_0 = \Delta[0]$, those that satisfy the Segal condition strictly are equivalent to simplicial abelian monoids, whereas those satisfying it only up to homotopy, simply called $\Gamma$-spaces, can be regarded as infinite loop spaces, at least with an additional group-like condition.

In \cite{inverses}, we built on work of Bousfield \cite{bous} to encode group and abelian group structures, not by changing the diagram shape, but by modifying the Segal condition to one we call a Bousfield-Segal condition.\footnote{Note the change-of-diagram approach in that paper is incorrect; see \cite{inverseserror}.} In particular, for $\Gamma$-spaces this approach combines the Segal condition and the group-like condition into a single criterion.

Here, we use an approach from \cite{segalop} to form categories built from multiple copies of $\Gamma$ so that its diagrams of simplicial sets satisfying the up-to-homotopy Segal condition correspond to infinite loop spaces with a $G$-action, for a specified discrete group $G$. We expect that these ideas are well-known to experts, but we bring together the approaches of \cite{bous} and \cite{segalop} in a unified treatment.

In the last section, we discuss the global case, where the group $G$ varies and can be taken to be a simplicial group, analogously to what we do for group actions on Segal operads in \cite{segalop}. This approach builds on our previous work on group actions on categories and operads in \cite{catact}.

It should be noted that we only consider group actions here, and so the $\Gamma$-spaces which $G$-action which we obtain correspond to naive rather than genuine $G$-spectra. An approach to equivariant infinite loop space theory which utilizes the category of all finite pointed $G$-sets is used in \cite{san}, \cite{shim} and produces genuine $G$-spectra; some related work in progress is given in \cite{mmo}.  The case of global equivariant spectra, where the group varies, has been considered by Bohmann \cite{boh} and a comprehensive treatment is being developed by Schwede \cite{schwede}. While the structures we consider do not encode the full strength of genuine equivariant stable homotopy theory, we find the diagrammatic structures to be of interest. In the case of the action of a fixed group $G$, we need not restrict to finite groups but may consider any discrete group. In the global case, our approach suggests a method for understanding actions by up-to-homotopy simplicial groups as well as strict ones.

\section{Background}
In this section we give a review of Segal's category $\Gamma$ and some relevant results. By $\SSets$, we denote the category of simplicial sets, or functors $\Deltaop \rightarrow \Sets$ with the model structure equivalent to the usual model category of topological spaces. Here the simplicial indexing category $\Delta$ has objects the finite ordered sets and morphisms are the order-preserving maps.

\subsection{The category \texorpdfstring{$\Gamma$}{Gamma} and Segal maps}

We begin with the original definition of the category $\Gamma$, as given by Segal in \cite{segal}. Its objects are representatives of isomorphism classes of finite sets, and a morphism $S \rightarrow T$ is given by a map $\theta:S \rightarrow \mathcal P(T)$ such that $\theta(\alpha)$ and $\theta(\beta)$ are disjoint whenever $\alpha \neq \beta$. (Here $\mathcal P(T)$ is the power set of the set $T$.) The opposite category $\Gammaop$ has the following description of its own: the category with objects ${\bf n}=\{0, 1, \ldots, n\}$ for $n \geq 0$ and morphisms ${\bf m} \rightarrow {\bf n}$ such that $0 \mapsto 0$.

In $\Gamma^{op}$, there are maps $\varphi_{n,k} \colon {\bf n} \rightarrow {\bf 1}$ for any $1 \leq k \leq n$ given by, for any $0 \leq i \leq n$,
\[ \varphi_{n,k}(i) =
\begin{cases}
1 & \text{ if } i=k \\
0 & \text{ if } i \neq k.
\end{cases} \]
Given any functor $X \colon \Gamma^{op} \rightarrow \SSets$, we get induced maps $\varphi_{n,k} \colon X({\bf n}) \rightarrow X({\bf 1})$. The disjoint union
\[ \varphi_n = \coprod_{k=1}^n \varphi_{n,k} \] is called a \emph{Segal map}.

\begin{definition}
A $\Gamma$-\emph{space} $X$ is a functor $\Gamma^{op} \rightarrow \SSets$ such that $X({\bf 0}) \cong \Delta[0]$ and the Segal map $\varphi_n \colon X({\bf n}) \rightarrow (X({\bf 1}))^n$ is a weak equivalence of simplicial sets for each $n \geq 2$. If $X({\bf 0}) \cong \Delta[0]$ and if each Segal map $\varphi_n \colon X({\bf n}) \rightarrow (X({\bf 1}))^n$ is an
isomorphism, then $X$ is a \emph{strict} $\Gamma$-\emph{space}.
\end{definition}

Note that our definition differs from the original, in that even for non-strict $\Gamma$-spaces we require that $X({\bf 0})$ be a single point, rather than simply equivalent to one.

The following two results are due to Segal \cite{segal}.

\begin{prop} \label{simpabmon}
The category of strict $\Gamma$-spaces is equivalent to the category of simplicial abelian monoids.
\end{prop}

\begin{prop}
The category of $\Gamma$-spaces $X$ such that $\pi_0(X({\bf 1}))$ is a group is equivalent to the category of infinite loop spaces.
\end{prop}

Some constructions for $\Gamma$ can also be applied to $\Delta$. Segal defines a functor $\Delta \rightarrow \Gamma$ as follows. The object $[n]$ is sent to ${\underline n} = \{ 1,\dots, n\}$ for each $n \geq 0$, and a map $f \colon [m] \rightarrow [n]$ is sent to the map $\theta \colon {\underline m} \rightarrow {\underline n}$ given by $\theta(i)=\{j \in {\underline n} \mid f(i-1)<j \leq f(i)\}$.

We can also define Segal maps for simplicial spaces, or functors $X \colon \Deltaop \rightarrow \SSets$; here also we restrict to the case where $X_0 \cong \Delta[0]$. The Segal map $X_n \rightarrow (X_1)^n$ is induced by the maps $\alpha^{n,k}:[1] \rightarrow [n]$ in $\Delta$ for each $0 \leq k \leq n-1$, where $\alpha^{n,k}(0)=k$ and $\alpha^{n,k}(1)=k+1$. Then applying the functor $\Delta \rightarrow \Gamma$, these maps $\alpha^{n,k}$ are sent to $\varphi^{op}_{n,k}$.

\subsection{Bousfield-Segal maps}

Following an idea of Bousfield \cite{bous}, in $\Delta$ we define the maps $\gamma^{n,k}:[1] \rightarrow [n]$ given by $0 \mapsto 0$ and $1 \mapsto k+1$ for all $0 \leq k < n$. Restricting to the case where $X_0 \cong \Delta[0]$, we can define the \emph{Bousfield-Segal map} $X_n \rightarrow (X_1)^n$ induced by these maps.
When such an $X$ satisfies the condition that the Bousfield-Segal maps are weak equivalences of simplicial sets for all $n \geq 2$, we call it a \emph{Bousfield-Segal group}; the name is justified by the fact that it is equivalent to a simplicial group \cite{inverses}. A comparison to the usual Segal condition is given in figure~\ref{F:onlyonefigure}; the idea is to define a group in terms of the binary operation $a,b\mapsto ab^{-1}$.

\begin{figure}[b]
\def\svgwidth{3.5in}
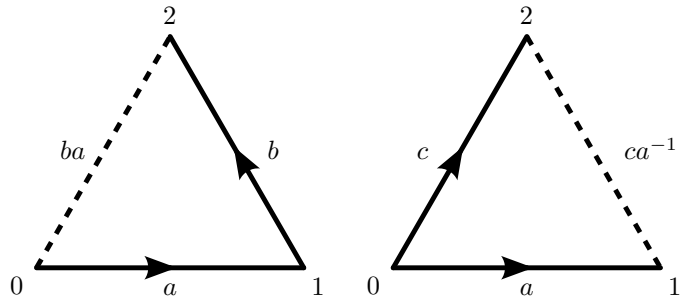
\caption{Left: Segal condition, Right: Bousfield-Segal condition}\label{F:onlyonefigure}
\end{figure}

Translating the maps $\gamma^{n,k}$ from $\Delta$ to $\Gamma$, we get maps $\delta^{n,k} \colon {\bf 1} \rightarrow {\bf n}$, one for each $1 \leq k \leq n$.

\begin{definition} \label{boussegal}
A \emph{strict Bousfield} $\Gamma$-\emph{space} is a functor $X:\Gamma^{op} \rightarrow \SSets$ such that $X({\bf 0}) \cong \Delta[0]$ and the maps $X({\bf n}) \rightarrow X({\bf 1})^n$ induced by the maps $\delta^{n,k}$ for all $1 \leq k \leq n$ are isomorphisms for all $n \geq 2$. Similarly, a \emph{(homotopy) Bousfield} $\Gamma$-\emph{space} has $X({\bf 0}) \cong \Delta[0]$ and the maps $X({\bf n}) \rightarrow X({\bf 1})^n$ weak equivalences of simplicial sets for $n \geq 2$.
\end{definition}

\begin{prop} \cite[7.2]{inverses} \label{simpabgp}
The category of strict Bousfield $\Gamma$-spaces is equivalent to the category of simplicial abelian groups.
\end{prop}

The following result was stated in \cite{inverses}; using the fact that the Bousfield-Segal condition gives a group structure, the group-like condition of Segal is satisfied.

\begin{prop}
The category of Bousfield $\Gamma$-spaces is equivalent to the category of infinite loop spaces.
\end{prop}

\subsection{Algebraic theories}

For strict $\Gamma$-spaces and Bousfield $\Gamma$-spaces, one method of comparison to their respective algebraic structures is via the machinery of algebraic theories.

\begin{definition}
An \emph{algebraic theory} $\mathcal T$ is a small category with finite products and objects denoted $T_n$ for $n \geq 0$. For each $n$, $T_n$ is equipped with an isomorphism $T_n \cong (T_1)^n$. Note in particular that $T_0$ is the terminal object in $\mathcal T$.
\end{definition}

\begin{definition}
Given an algebraic theory $\mathcal T$, a \emph{(strict simplicial)} $\mathcal T$-\emph{algebra} $A$ is a product-preserving functor $A:\mathcal T \rightarrow \mathcal {SS}ets$.
\end{definition}

Here, ``product-preserving" means that for each $n \geq 0$ the canonical map
\[ A(T_n) \rightarrow A(T_1)^n, \]
induced by the $n$ projection maps $T_n \rightarrow T_1$, is an isomorphism of simplicial sets. In particular, $A(T_0)$ is the one-point simplicial set $\Delta [0]$. For a given algebraic theory $\mathcal T$, we denote by $\Algt$ the category of $\mathcal T$-algebras.

The proofs of Propositions \ref{simpabmon} and \ref{simpabgp} can be established by comparing the categories of simplicial abelian monoids and simplicial abelian groups to the categories of algebras over the theories $\mathcal T_{\mathcal{AM}}$ of abelian monoids and $\mathcal T_{\mathcal{AG}}$ of abelian groups, respectively. Each of these categories has as objects the finitely generated free objects in the appropriate category.

\section{Actions of a fixed group on \texorpdfstring{$\Gamma$}{Gamma}-spaces}

In this section, we give a diagrammatic description of infinite loop spaces with a $G$-action, for a given group $G$. The diagram we give is essentially a wedge product of copies of $\Gamma$, indexed by the elements of the group $G$. We assume throughout that a discrete group $G$ is fixed.

Generalizing the case of strict $\Gamma$-spaces, we begin by considering abelian monoids with a $G$-action. There is an algebraic theory $\Tgam$ of abelian monoids with $G$-action; it is the full subcategory of the category of abelian monoids with objects, for each $n \in \mathbb N$, given by $G \times F_n$, where $F_n$ denotes the free abelian monoid on $n$ generators.

Define $G \Gamma^{op}$ to be the category with objects ${\bf n}_G = \vee_{g \in G} {\bf n}_g$, where ${\bf n}_g=\{0_g, 1_g, \ldots, n_g \}$. The morphisms are generated by:
\begin{itemize}
\item morphisms $\vee_{g \in G} f \colon \vee_{g \in G} {\bf n}_g \rightarrow \vee_{g \in G} {\bf m}_g$ where each $f \colon {\bf n}_g \rightarrow {\bf m}_g$ is given by the same morphism $f \colon {\bf n} \rightarrow {\bf m}$ of $\Gamma^{op}$, and

\item automorphisms given by a $G$-action, so, for $g \in G$, $g \cdotp \colon {\bf n}_G \rightarrow {\bf n}_G$ given by $g \cdotp k_h =k_{gh}$.
\end{itemize}

\begin{example}
Consider the case where $G=\mathbb Z/2=\{0,1\}$. Then the objects ${\bf n}_{\mathbb Z/2}$ can be thought of as pointed sets
\[ \{n_0, \ldots, 2_0, 1_0, 0, 1_1, 2_1, \ldots, n_1 \}. \] The action of $\mathbb Z/2$ sends each $k_0$ to $k_1$ and vice versa.
\end{example}

Consider functors $X \colon G \Gamma^{op} \rightarrow \SSets$ such that $X({\bf 0}_G) \cong \Delta[0]$. The Segal maps for $\Gamma$ induce Segal maps $X({\bf n}_G) \rightarrow (X({\bf 1}_G))^n$ for $G\Gamma$, where the behavior on each copy of $\Gamma$ in $G\Gamma$ is the same. Similarly, the Bousfield-Segal maps for $\Gamma$ induce Bousfield-Segal maps for $G\Gamma$.

The proof of the following proposition generalizes the one given in \cite[7.1]{inverses}.

\begin{prop}
The category of functors $X \colon \GGammaop \rightarrow \SSets$ satisfying the strict Segal condition is equivalent to the category of simplicial abelian monoids equipped with a $G$-action.
\end{prop}

\begin{proof}
First recall that the category of simplicial abelian monoids with a group action is equivalent to the category $\Alg^{\Tgam}$ of strict algebras over $\Tgam$, so it suffices to establish an equivalence with $\Alg^{\Tgam}$.

In $\GGammaop$, there are projection maps $p_{n,i,G} \colon {\bf n}_G \rightarrow {\bf 1}_G$ where $p_{n,i,G}(k_g) = 1_g$ if $k=i$ and 0 otherwise. The natural functor $f \colon \GGammaop \rightarrow \Tgam$, given by ${\bf n}_G \mapsto G\times F_n$, preserves these projection maps.

A strict $G\Gamma$-space $X \colon G\Gamma^{op} \rightarrow \SSets$ is determined by each simplicial set $X({\bf n}_G)$, the projection maps $X({\bf n}_G) \rightarrow X({\bf 1}_G)$, and the map $X({\bf 2}_G) \rightarrow X({\bf 1}_G)$ which is the image of the map ${\bf 2}_G \rightarrow {\bf 1}_G$ given by $0 \mapsto 0$ and $1_g, 2_g \mapsto 1_g$ for each $g \in G$. In particular, by induction the map $X({\bf 2}_G) \rightarrow X({\bf 1}_G)$ induces all maps $X({\bf n}_G) \rightarrow X({\bf 1}_G)$ arising from the projections ${\bf n}_G \rightarrow {\bf 1}_G$ given by $0 \mapsto 0$ and $i_g \mapsto 1_g$ for each $1 < i \leq n$ and $g \in G$. Then the structure of a strict $G\Gamma$-space gives the space $X({\bf 1}_G)$ the structure of an abelian monoid (with multiplication map given by $X({\bf 2}_G) \rightarrow X({\bf 1}_G)$ as above), with a $G$-action given by the map $G \times X({\bf 1}_G) \rightarrow X({\bf 1}_G)$ induced from $G \times {\bf 1}_G \rightarrow {\bf 1}_G$ defined by $(g, 1_h) \mapsto 1_{gh}$.

In particular, from the simplicial set $X({\bf 1}_G)$ we can produce a $\Tgam$-algebra $tX : \Tgam \to \SSets$ given by $G\times F_n \mapsto X({\bf n}_G)$ for each $n$.
The projection maps agree with those of $X$, that is $X(p_{n,i,G})$ coincides with the map coming from the $i$th projection $G\times F_n \to G\times F_1$ in
\[
tX(G\times F_n) = X(\mathbf n_G) \to  X(\mathbf 1_G) = tX(G\times F_1).
\]
Restricting $tX$ along $f \colon \GGammaop \rightarrow \Tgam$ produces the original $G\Gamma$-space $X$. Denoting this restriction map by $f^*$, we have shown that the functors $t$ and $f^*$ are inverse to one another.
\end{proof}

\begin{cor}
The category of functors $X \colon \GGammaop \rightarrow \SSets$ satisfying the strict Bousfield-Segal condition is equivalent to the category of simplicial abelian groups equipped with a $G$-action.
\end{cor}

\begin{prop}
A functor $X \colon \GGammaop \rightarrow \SSets$ satisfying the (homotopy) Segal condition determines a spectrum with a $G$-action.
\end{prop}

\begin{proof}
Suppose that $X$ is a $G\Gamma$-space. Then $X({\bf 0}_G)$ is a point and the Segal maps $X({\bf n}_G) \rightarrow X({\bf 1}_G)^n$ are weak equivalences for $n \geq 2$. To construct the classifying space of a $\Gamma$-space $X$, Segal first defines a bi-$\Gamma$-space $\widehat X({\bf m},{\bf n})= X({\bf m} \wedge {\bf n})$. Then the classifying space $BX$ is obtained by composing the adjoint of the composite map
\[ \Deltaop \times \Gammaop \rightarrow \Gammaop \times \Gammaop \rightarrow \SSets \]
with ordinary geometric realization of a simplicial space to get
\[ \Gammaop \rightarrow \SSets^{\Deltaop} \rightarrow \SSets. \]

Since we require that $X({\bf 0})$ be a point, we should verify that this condition still holds for the classifying space. The space $BX({\bf 0})$ is the geometric realization of the functor ${\bf n} \mapsto X({\bf 0} \wedge {\bf n})$. Since $X({\bf 0} \wedge {\bf n}) = X({\bf 0})$, we have that $BX({\bf 0})$ is the geometric realization of the constant simplicial space of a point.

Now we consider the case at hand, where we have functors $\Deltaop \rightarrow \Gammaop$ and $e \colon \Gammaop \rightarrow G\Gammaop$. If $X$ is a $G\Gamma$-space, we can perform the same construction as above to get a bi-$G\Gamma$-space defined by $({\bf n}_G, {\bf m}_G) \mapsto X({\bf n}_G \wedge {\bf n}_G)$, which we denote by $\widehat X$. Observing that $\widehat{e^*X} = (e \times e)^* \widehat X$, we have the commutative diagram
\[ \xymatrix{\Deltaop \times \Gammaop \ar[dd] \ar[r] & \Gammaop \times \Gammaop \ar[dd] \ar[dr]^{\widehat{e^*X}} & \\
&& \SSets. \\
\Deltaop \times G\Gammaop \ar[r] & G\Gammaop \times G\Gammaop \ar[ur]_{\widehat X} &} \]
Taking adjoints, we have the commutative diagram
\[ \xymatrix{\Gammaop \ar[dd] \ar[rd] && \\
& \SSets^{\Deltaop} \ar[r]^{|-|} & \SSets. \\
G\Gammaop \ar[ur] &&} \]
We then see that $B(e^*X)$ naturally has the structure of a $G\Gamma$-space, and we denote it $BX$.

Iterating this construction, we obtain a sequence of spaces
\[ X({\bf 1}_G), BX({\bf 1}_G), B^2X({\bf 1}_G), \ldots \]
which is a spectrum since $X$ is a $\Gamma$-space. Each $B^nX({\bf 1}_G)$ here is also equipped with a $G$-action, so it remains to show that the spectrum structure maps preserve this action.

In the ordinary case of $\Gamma$-spaces, Segal shows that the 1-skeleton of $X$, $\sk_1|X|$, is homotopy equivalent to $\Sigma X({\bf 1})$; the argument can be given using a diagram such as the following:
\[ \xymatrix{(\ast \amalg (\Delta^1 \times X_1))/\sim \ar[d]_= & ((\Delta^0 \times X_0) \amalg (\Delta^1 \times X_1))/ \sim \ar[d]^= \ar[l] & \\
\Sigma X({\bf 1}) & \sk_1|X| \ar[l]_\simeq \ar[r] & |X|.} \]
Since $X({\bf 0})$ is contractible, the upper horizontal map is a homotopy equivalence, and therefore so is the indicated lower horizontal map. The spectrum structure map is given by choosing a homotopy inverse to this map and postcomposing with the inclusion $\sk_1|X| \rightarrow |X|$.

In our case, we have imposed the more restrictive condition that $X({\bf 0}_G)= \ast$. It follows that $\Sigma X({\bf 1}_G)$ is isomorphic to $\sk_1|X|$ (not just weakly equivalent), so the structure map is given by the inclusion $\Sigma X({\bf 1}_G) \cong \sk_1|X| \rightarrow |X| = BX({\bf 1}_G)$, which is $G$-equivariant.

\end{proof}

\begin{cor}
If $X \colon G\Gamma^{op} \to \SSets$ satisfies the Bousfield-Segal condition, then $X(\mathbf{1})$ is an infinite loop space with $G$-action.
\end{cor}

\section{Varying group actions on \texorpdfstring{$\Gamma$}{Gamma}-spaces}

In this section, we consider a diagram which encodes group actions on $\Gamma$-spaces, where the groups may vary.  The emphasis of this section is somewhat different from the previous one, in that our goal is to find a diagram and Segal conditions which encode both the $\Gamma$-space and the acting group.  We give some preliminary results in this section, in particular when the Segal maps are isomorphisms; a full treatment would require substantially more subtle tools.

Note that we have chosen, for simplicity, to work with actions on $\Gamma$-spaces, in contrast with the previous section where we studied actions on Bousfield $\Gamma$-spaces; the interested reader may modify Definition \ref{dgspace} (2) to ensure the relevant group-like condition holds.

We define a category $\mathcal L$ whose objects are functors $\mathcal C \rightarrow \Fin$, where $\mathcal C$ is a category (which can vary) with finitely many objects and $\Fin$ denotes (a skeleton of) the category of finite sets. Let $\mathcal P \colon \Fin \rightarrow \Fin$ be the functor taking a finite set to its power set. Define a morphism in $\mathcal L$ from $F \colon \mathcal C \rightarrow \Fin$ to $G \colon \mathcal D \rightarrow \Fin$ to be a pair $(f,\eta)$ consisting of:
\begin{itemize}
\item a functor $f \colon \mathcal C \rightarrow \mathcal D$, and

\item a natural transformation $\eta \colon F \Rightarrow \mathcal PGf$,
\end{itemize}
subject to the condition that if $c,c'\in (\ob f)^{-1}(d)$, $x\in F(c)$, and $y\in F(c')$, then $\eta_c (x), \eta_{c'}(y) \in \mathcal PG(d)$ have empty intersection unless $c=c'$ and $x=y$.

Given such a functor $F$, define its \emph{total space} to be
\[ \tau(F) = \coprod_{c \in \ob \mathcal C} F(c) \]
which is precisely the finite set that $\mathcal C$ acts upon. Thus, the above condition for morphisms $(f,\eta) \colon F \rightarrow G$ in $\mathcal L$ can be alternatively rephrased to say that the composite map
\[ \theta \colon \xymatrix@1{\tau(F) \ar[r]^-{\coprod \eta_c} & \tau(\mathcal PGf) \ar[r] & \tau(\mathcal PG) = \coprod_{d \in \ob(\mathcal D)} \mathcal PG(d) \ar[r] & \mathcal P(\tau(G))} \]
must satisfy that $\theta(x)$ and $\theta(y)$ are disjoint whenever $x \neq y$. The fact that composition in this category is well-defined can be established using the monadic structure of the endofunctor $\mathcal P \colon \Fin \to \Fin$.

Observe that if $\mathcal C$ and $\mathcal D$ are both the trivial category $*$ with one object and an identity morphism only, then we recover the category $\Gamma$. One can also see that, for each discrete group $G$, the category $G\Gamma$ as described in the previous section can be identified with the subcategory of $\mathcal L$ whose objects are functors $F\colon G^{op} \to \Fin$ so that $F(*)$ is a \emph{free} $G$-set, with morphisms restricted to those for which $f \colon G^{op} \to G^{op}$ is the identity map.

Consider the functor $C_{0,-} \colon \Gamma \rightarrow \mathcal L$ which sends an object $S$ of $\Gamma$ to the object $C_{0,S} \colon [0] \rightarrow \Fin$ satisfying $C_{0,S}(0) = S$.
Observe that
\[ \Hom_\Gamma(S,T) = \Hom_{\mathcal L} (C_{0,S}, C_{0,T}). \]
Similarly, consider the functor $C_{-,0} \colon \Delta \rightarrow \mathcal L$ which sends an object $[n]$ of $\Delta$ to the constant functor $C_{n,0} \colon [n] \rightarrow \Fin$ with $C_{n,0}(i) = \varnothing$ for all $i$.
Since there is a unique set map $\varnothing \rightarrow \mathcal P (\varnothing)$, we have
\[ \Hom_\Delta([n],[m]) = \Hom_{\mathcal L} (C_{n,0}, C_{m,0}). \]

For $m,n$ nonnegative integers, let $C_{m,n} \colon [m] \to \Fin$ be the constant functor on the set $\underline n = \{1,\dots, n\}$. Let $\Deltaactgamma$ denote the full subcategory of $\mathcal L$
with
\[ \ob (\Deltaactgamma) = \left\{ C_{m,n} \right\}_{m,n\geq 0}. \]
If $X \colon (\Deltaactgamma)^{op} \rightarrow \SSets$ is a functor, we use the notation
\[ X^\Gamma = X \circ C_{0,-} \colon \Gammaop \rightarrow \SSets \]
and
\[ X^\Delta = X \circ C_{-,0} \colon \Deltaop \rightarrow \SSets. \]

There is a unique morphism $\ell_1 = (f,\eta) \colon C_{m,0} \to C_{m,n}$ with $f = \id_{[m]}$.
Also consider the morphism
\[ \ell_2 = (f,\eta) \colon C_{0,n} \to C_{m,n} \]
such that $f \colon [0] \to [m]$ takes the object $0$ to $m$ and the natural transformation
$\eta \colon C_{m,0} \Rightarrow \mathcal P C_{m,n} f$ is given by
\begin{align*}
\eta_0 \colon C_{0,n}(0) = \underline n & \to \mathcal P (\underline n) = C_{m,n}(m) \\
i &\mapsto \{ i\}.
\end{align*}
If $X \colon (\Deltaactgamma)^{op} \rightarrow \SSets$ is any functor, then there is a map
\[ \ell_1 \times \ell_2 \colon X(C_{m,n}) \to X(C_{m,0}) \times_{X(C_{0,0})} X(C_{0,n}) = X^\Delta_m \times_{X^\Delta_0} X^\Gamma(\mathbf n). \]

We use these maps to make the following definition.

\begin{definition} \label{dgspace}
A $\Deltaactgamma$-\emph{space} is a functor $X \colon (\Deltaactgamma)^{op} \rightarrow \SSets$ satisfying:
\begin{enumerate}
\item the space $X(C_{0,0}) = X^\Delta_0 = X^\Gamma(\mathbf 0)$ is a single point;

\item the induced Segal maps $X^\Gamma(\mathbf n) \rightarrow X^\Gamma(\mathbf 1)^n$ are weak equivalences;

\item \label{bousfield-segal} the induced Bousfield-Segal maps $X^\Delta_n \rightarrow (X^\Delta_1)^n$ are weak equivalences;

\item \label{action split} the maps $\ell_1 \times \ell_2 \colon X(C_{m,n}) \to X^\Delta_m \times X^\Gamma(\mathbf n)$ are weak equivalences for all $m,n$.
\end{enumerate}
\end{definition}

Also consider the morphism
\[ \ell_3 = (f,\eta) \colon C_{0,n} \to C_{m,n} \]
with $f(0) = 0$ and
\begin{align*}
\eta_0 \colon C_{0,n}(0) = \underline n & \to \mathcal P (\underline n) = C_{m,n}(0) \\
i &\mapsto \{ i\}.
\end{align*}
This structure,
\begin{equation} \label{EQaction} \includegraphics{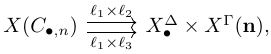} \end{equation}
should be interpreted as giving a higher-order action of $X^\Delta$ on $X^\Gamma(\mathbf n)$, in the spirit of \cite[5.1]{prezma2}.

\begin{remark}
An explanation is in order regarding this definition. Consider, for the moment, the setting where we require isomorphisms throughout, rather than just weak equivalences. Then $X^\Gamma(\mathbf 1)$ is a simplicial abelian monoid and $X^\Delta_1$ is a simplicial group. Condition \eqref{action split} allows us to define an action of $X^\Delta_1$ on $X^\Gamma(\mathbf 1)$, via the diagram
\[ \begin{tikzcd}
X(C_{1,1}) \dar{\ell_3} \rar{\ell_1 \times \ell_2}[swap]{\cong} &
X^\Delta_1 \times X^\Gamma(\mathbf 1) \\
X^\Gamma(\mathbf 1).
\end{tikzcd} \]
One can check that this binary operation is actually an action of $X^\Delta_1$ on $X^\Gamma(\mathbf 1)$.
\end{remark}

Let us now consider how to obtain a spectrum from such a structure, following Segal's construction for $\Gamma$-spaces.

There is a bifunctor
\[ \mu \colon \Gamma \times \mathcal L \to \mathcal L \]
which, on objects, takes $(S, F \colon \mathcal{C} \to \Fin)$ to the functor
\[ S \times F \colon \mathcal{C} \rightarrow \Fin \]
given by $i \mapsto S \times F(i)$. If $(f,\eta) \colon F \to G$ is a morphism of $\mathcal L$ and $S$ is an object of $\Gamma$, we get a new morphism $(f, S\times \eta)$ of $\mathcal L$ by setting
\[ (S\times \eta)_i \colon S \times F(i) \overset{\eta_i}\to S \times \mathcal P(Gf(i)) \to \mathcal P(S\times Gf(i)). \]
If $F \colon \mathcal{C} \to \Fin$ is an object of $\mathcal L$ and $\theta \colon S \to \mathcal PT$ is in $\Hom_{\Gamma} (S,T)$,
we get a morphism $(\id_{\mathcal{C}}, \theta \times F) \colon S\times F \to T \times F$ in $\mathcal L$ by the composite
\[ (\theta \times F)_i \colon S \times F(i) \overset{\theta}\to \mathcal P(T) \times F(i) \to \mathcal P(T \times F(i)). \]
Note that $\mu$ induces a bifunctor $\Gamma \times (\Deltaactgamma) \to \Deltaactgamma$ with $\mu( \underline p, C_{m,n}) = C_{m,np}$.

Suppose $X$ is a $\Deltaactgamma$-space; we are now able to define the classifying space $BX$.
The composite
\[ \Deltaop \times (\Deltaactgamma)^{op} \to \Gamma^{op} \times (\Deltaactgamma)^{op} \overset{\mu}\to (\Deltaactgamma)^{op} \overset{X}\rightarrow \SSets \]
has adjoint
\[ bX \colon (\Deltaactgamma)^{op} \to \SSets^{\Deltaop}; \]
composing with geometric realization $\SSets^{\Deltaop} \to \SSets$ gives a new functor
\[ BX \colon (\Deltaactgamma)^{op} \to \SSets.\]

\begin{lemma}\label{lemma Gamma interchange}
Let $X \colon (\Deltaactgamma)^{op} \to \SSets$ be a functor. Then
\[ B(X^\Gamma) \cong (BX)^\Gamma.\]
\end{lemma}

\begin{proof}
We have the following commutative diagram:
\[ \begin{tikzcd}
\Deltaop \times \Gamma^{op} \dar{\id \times C_{0,-}} \rar & \Gamma^{op} \times \Gamma^{op} \dar{\id \times C_{0,-}} \rar{\mathbf{m}, \mathbf{n} \mapsto \mathbf{m} \wedge \mathbf{n}} & \Gamma^{op} \dar{C_{0,-}} \arrow{dr}{X^\Gamma} \\
\Deltaop \times (\Deltaactgamma)^{op} \rar & \Gamma^{op} \times (\Deltaactgamma)^{op} \rar[swap]{\mu^{op}} & (\Deltaactgamma)^{op} \rar[swap]{X} & \SSets.
\end{tikzcd} \]
Taking adjoints yields the commutative triangle on the left in the diagram:
\[ \begin{tikzcd}
\Gamma^{op} \dar[swap]{C_{0,-}} \arrow{dr} \arrow[bend left]{drr}{B(X^\Gamma)} \\
(\Deltaactgamma)^{op} \rar{bX} \arrow[bend right]{rr}[swap]{BX} & \SSets^{\Deltaop} \rar{|-|} & \SSets.
\end{tikzcd} \]
Thus $(BX)^\Gamma = BX \circ C_{0,-} = B(X^\Gamma)$.
\end{proof}

A consequence of this lemma is that if $X$ is a $\Deltaactgamma$-space, then $BX$ satisfies the first two conditions of Definition \ref{dgspace}.

\begin{lemma}
Let $X \colon (\Deltaactgamma)^{op} \to \SSets$ be a functor. Then
\[ X^\Delta \cong (BX)^\Delta. \]
\end{lemma}

\begin{proof}
We show that $bX(C_{m,0})$ is a constant simplicial space. Since $S \times \varnothing = \varnothing$ for any set $S$, the functor
\[ \mu(-,C_{m,0}) \colon \Gamma \to \Deltaactgamma \]
is constant with value $C_{m,0}$. Now, for any $n$ we have
\[ (bX(C_{m,0}))_n = X(\mu(\underline n,C_{m,0})) = X(C_{m,0}) = X^\Delta_m. \]
Thus $(BX)^\Delta_m = |bX(C_{m,0})| = X^\Delta_m$.
\end{proof}

An immediate consequence is that $X$ satisfies condition \eqref{bousfield-segal} of Definition \ref{dgspace} if and only if $BX$ does.

\begin{theorem}
If $X$ is a $\Deltaactgamma$-space, then so is $BX$.
\end{theorem}

\begin{proof}
In light of the previous two lemmas, we need only show that condition \eqref{action split} holds for $BX$.
We will show that the map of simplicial spaces
\begin{equation}\label{simplicial spaces map}
\ell_1 \times \ell_2 \colon bX(C_{m,n}) \to bX(C_{m,0}) \times bX(C_{0,n})
\end{equation}
is a levelwise weak equivalence, which implies that we have a weak equivalence after geometric realization.
Notice that $\mu(\underline p, C_{m,n}) \cong C_{m,np}$.
For each $p$ we then have the diagram
\[ \begin{tikzcd}
(bX(C_{m,n}))_p \rar{\cong} \dar & X(C_{m,np}) \dar{\simeq} \\
(bX(C_{m,0}))_p \times (bX(C_{0,n}))_p \rar{\cong} & X(C_{m,0}) \times X(C_{0,np}),
\end{tikzcd} \]
where the map on the right is a weak equivalence by assumption. Hence, the map on the left is a weak equivalence.
We thus have established that the map \eqref{simplicial spaces map} is a levelwise weak equivalence, hence its geometric realization
\[ BX(C_{m,n}) \to BX(C_{m,0}) \times BX(C_{0,n}) \]
is a weak equivalence by \cite[15.11.11]{hirschhorn}, so we see that condition \eqref{action split} of Definition \ref{dgspace} holds for $BX$.
\end{proof}

Notice that the sequence $X^\Gamma(\mathbf 1), BX^\Gamma(\mathbf 1), B^2X^\Gamma(\mathbf 1), \dots$ forms a spectrum of spaces $\mathbf B X^\Gamma$ as in \cite{segal}, using Lemma \ref{lemma Gamma interchange}.
As mentioned above, the map of simplicial spaces
\[ A_\bullet = X(C_{\bullet,1}) \to X^\Delta \times X^\Gamma(\mathbf 1) \]
should be interpreted (cf. \eqref{EQaction}) as providing an action of the Bousfield-Segal group $X^\Delta$ on $X^\Gamma(\mathbf 1) = A_0$, using the higher action map $d_1\dots d_m : A_m \to A_0$ :\[ \begin{tikzcd}
A_m \dar[swap]{d_1\dots d_m} \rar{\cong} &
X(C_{m,1}) \dar{\ell_3} \rar{\ell_1 \times \ell_2}[swap]{\simeq} &
X^\Delta_m \times X^\Gamma(\mathbf 1) \\
A_0 \rar{\cong} &
X^\Gamma(\mathbf 1).
\end{tikzcd} \]
This action is compatible with the $B$ construction as follows.  Let $i_1 \colon \sk_1 \Delta \to \Delta$ be the full subcategory inclusion, where $\ob(\sk_1 \Delta) = \{ [0], [1] \}$, and write $\sk_1 BX$ for the composite
\[ \Deltaactgamma^{op} \overset{bX}\longrightarrow \SSets^{\Deltaop} \overset{i_{1*}}\longrightarrow \SSets^{\sk_1\Delta^{op}} \overset{i_1^*}\longrightarrow \SSets^{\Deltaop} \overset{|-|}\longrightarrow \SSets \]
(see \cite[\S IV.3.2]{gj}).
Compatibility is given by the diagram
\[ \begin{tikzcd}
\Sigma X^\Gamma(\mathbf 1) \arrow{d}{\cong} &
&
X_m^\Delta \times \Sigma X^\Gamma(\mathbf 1) \arrow{d}{\cong} \\
\sk_1 BX (C_{0,1}) \dar &
\sk_1 BX (C_{m,1}) \arrow{r}{\ell_1 \times \ell_2} \arrow{l}[swap]{\ell_3} \dar &
\sk_1 BX(C_{m,0}) \times \sk_1 BX(C_{0,1}) \dar \\
BX (C_{0,1}) \dar{=} &
BX (C_{m,1}) \arrow{r}{\ell_1 \times \ell_2} \arrow{l}[swap]{\ell_3} &
BX(C_{m,0}) \times BX(C_{0,1}) \dar{=} \\
BX^\Gamma(\mathbf 1) &
&
X_m^\Delta \times BX^\Gamma(\mathbf 1).
\end{tikzcd} \]
Thus it seems reasonable to say that the Bousfield-Segal group $X^\Delta$ acts levelwise on the spectrum $\mathbf B X^\Gamma$.  We expect that a proof would require techniques such as those of \cite{prasma}.

\begin{remark}
Observe that there are still several questions to be answered here.  First, in order to provide a precise definition of an action of a Bousfield-Segal group on a space, we first need to address a more fundamental question, namely the close relationship between Segal groups as defined by Prasma \cite{prasma} and Bousfield-Segal groups as we have defined them here, following \cite{inverses}.  We expect that the two notions are equivalent, but have not given a proof.

Second, we would like to have an explicit way to regard a $G\Gamma$-space in the sense of the previous section as a special case of our construction here.
In particular, can one construct a functor
\[
\SSets^{G\Gamma^{op}} \to \SSets^{\Deltaactgamma^{op}}
\]
which takes a functor satisfying the (homotopy) Segal condition to a $\Deltaactgamma$-space?
\end{remark}

\end{document}